\documentclass[10pt, a4paper]{amsart}
\usepackage{amssymb, amsmath, amsfonts, mathrsfs, amsthm}
\usepackage[utf8]{inputenc}
\usepackage{lmodern}
\usepackage[T1]{fontenc}
\usepackage[a4paper,includeheadfoot,margin=2.54cm]{geometry}

\DeclareMathOperator{\Exc}{\mathrm{Exc}}

\DeclareMathOperator{\LL}{\mathscr{L}}

\DeclareMathOperator{\OO}{\mathscr{O}}
\DeclareMathOperator{\EE}{\mathscr{E}_{\mathbf{a}}}
\DeclareMathOperator{\sing}{\mathrm{sing}}
\DeclareMathOperator{\Sing}{\mathrm{Sing}}
\DeclareMathOperator{\Gk}{\mathbb{G}_k}
\DeclareMathOperator{\Gpk}{\mathbb{G}^{\prime}_k}
\DeclareMathOperator{\reg}{\mathrm{reg}}
\DeclareMathOperator{\Sym}{\mathrm{Sym}}

\begin{document} 

\title{On negativity of total $k$-jet curvature and ampleness of the canonical bundle}
\date{}
\author{Aleksei Golota}

\theoremstyle{definition}

\newtheorem{thm}{Theorem}[subsection]
\newtheorem{defi}{Definition}[subsection]
\newtheorem{prop}{Proposition}[subsection]
\newtheorem{rmk}{Remark}[subsection]
\newtheorem{cor}{Corollary}[subsection]
\newtheorem{conj}{Conjecture}
\newtheorem{que}{Question}[subsection]
\newtheorem{assum}{Assumption}[subsection]

\begin{abstract}
A celebrated conjecture of Kobayashi and Lang says that the canonical line bundle $K_X$ of a Kobayashi hyperbolic compact complex manifold $X$ is ample. In this note we prove that $K_X$ is ample if $X$ is projective and satisfies a stronger condition of nondegenerate negative total $k$-jet curvature. The main idea is to apply recent results on positivity of direct image sheaves to the Demailly-Semple tower in order to produce pluridifferentials on $X$.
\end{abstract}

\maketitle

\section{Introduction}

A complex manifold $X$ is called {\em Kobayashi hyperbolic} if  an intrinsic pseudometric $d_K$, introduced by S. Kobayashi (see \cite{Kob70}, IV. 1) is a metric on $X$. Assume from now on that $X$ is compact, for instance projective. Thanks to the Brody criterion \cite{Brd78}, Kobayashi hyperbolicity of $X$ is equivalent to non-existence of {\em entire curves}, i. e. nonconstant holomorphic maps $f \colon \mathbb{C} \to X$. 

One the most interesting and challenging tasks in complex geometry is to describe the precise relations between Kobayashi hyperbolicity and birational geometry of $X$. Recall the two celebrated conjectures.

\begin{conj}\label{conj:Kob}{(S. Kobayashi, \cite{Kob70}, S. Lang \cite{Lan86})} Let $X$ be a Kobayashi hyperbolic compact complex manifold. Then the canonical line bundle $K_X$ is ample.
\end{conj}

\begin{conj}\label{conj:GGL}{(M. Green -- P. Griffiths, S. Lang \cite{GG80, Lan86})} Let $X$ be a variety of general type. Then there exists a proper subvariety $Z \subsetneq X$ such that for every entire curve $f \colon \mathbb{C} \to X$ the image $f(\mathbb{C})$ lies in $Z$.
\end{conj}

We refer an interested reader to e. g. \cite{GG80, McQ98, DMR10} as well as to expository works \cite{Bru99, Dem12, Siu04, Voi03} for the progress towards the Green-Griffiths-Lang conjecture. As for Conjecture \ref{conj:Kob}, it is obviously true in dimension one and follows from Kodaira-Enriques classification in dimension two. A possible strategy (which works in dimension $3$) is to argue by contradiction using the Iitaka fibration and (conjectural) non-hyperbolicity of Calabi-Yau varieties. In higher dimensions, however, both the abundance conjecture and non-hyperbolicity of Calabi-Yau varieties seem to be at present beyond reach. Thus a natural attempt would be to derive ampleness of $K_X$ under stronger assumptions on $X$. For example, it is classically known \cite{GRe65} that if $X$ has a Hermitian (or, more generally, a Finsler) metric with negative holomorphic sectional curvature then $X$ is Kobayashi hyperbolic. In this case we have a major result established recently in \cite{WY16, WY16b, DT16}.

\begin{thm}\label{thm:WuYau}{(D. Wu -- S. T. Yau)} If $X$ carries a K\"ahler metric with (quasi-)negative holomorphic sectional curvature then $K_X$ is ample.
\end{thm}

Note that negativity of the holomorphic sectional curvature is a strictly stronger assumption than hyperbolicity. On the other hand, we can generalize it in the following way. The K\"ahler metric $\omega$ induces a smooth Hermitian metric on the hyperplane line bundle $\OO_{\mathbb{P}(T_X)}(1)$ over the projectivized bundle $\pi \colon \mathbb{P}(T_X) \to X$ of {\em lines} in $T_X$. Negativity of the holomorphic sectional curvature $\Theta_{\omega}(v\otimes v)$ is equivalent to negativity of the curvature form $ \Theta_{h_1}(\OO_{\mathbb{P}(T_X)}(1))$ along the subbundle $V_1 \subset T_{\mathbb{P}(T_X)}$ of vectors projected to the line $\OO_{\mathbb{P}(T_X)}(-1)$ by $\pi_*$. We can proceed inductively starting from $(X_0, V_0) = (X, T_X)$ and construct the $k$-th stage of {\em the Demailly -- Semple tower} by putting $$ (X_k, V_k) := \left(\mathbb{P}(V_{k-1}), (\pi_{k-1,k})_*^{- 1}\OO_{\mathbb{P}(V_{k-1})}( - 1)\right)$$ where $\pi_{k-1, k} \colon X_k \to X_{k-1}$ is the projection. The varieties $X_k$ are endowed with hyperplane line bundles $\OO_{X_k}(1)$. We say that $(X, T_X)$ has {\em negative $k$-jet curvature} if there exists a (possibly singular) Hermitian metric $h_k$ on $\OO_{X_k}(1)$ such that $$ \Theta_{h_k}(\OO_{X_k}(1))|_{V_k} \geqslant \varepsilon \omega|_{V_k} $$ in the sense of currents for some $\varepsilon > 0$ and a metric $\omega$ on $X_k$. It can be shown that $k$-jet negativity becomes weaker as $k$ increases and together with {\em nondegeneracy} condition (see Definition \ref{defi:KJet} below) implies Kobayashi hyperbolicity. The examples in \cite{Dem97}, Theorem 8.2 show that for every $k_0 \in \mathbb{N}$ there are Kobayashi hyperbolic surfaces which do not admit a $k$-jet negative metric for $k \leqslant k_0$. Nevertheless, the following has a chance to be true.

\begin{conj}{(J.-P. Demailly \cite{Dem97})} The variety $X$ is Kobayashi hyperbolic if and only if there exists $k \in \mathbb{N}$ such that $(X, T_X)$ has negative $k$-jet curvature.
\end{conj}

Therefore it is plausible to expect the following generalization of Theorem \ref{thm:WuYau}.

\begin{que}\label{que:KJet} Assume that $(X, T_X)$ has nondegenerate negative $k$-jet curvature for some $k \in \mathbb{N}$. Is the canonical line bundle $K_X$ ample?
\end{que}

Observe that in \ref{que:KJet} the curvature of $\OO_{X_k}(1)$ is only {\em partially} positive, so it does not {\em a priori} say anything about the existence of sections. The methods of \cite{WY16, WY16b} also do not seem to generalize directly to the case $k > 1$. Nevertheless, we are able to answer the Question \ref{que:KJet} under a strong (in fact, strongest possible, see \cite{Dem97}, Theorem 6.8 (iii)) positivity assumption on the curvature of $\OO_{X_k}(1)$. Our main theorem states as follows.

\begin{thm}\label{thm:Main} Assume that $(X, T_X)$ has nondegenerate negative {\em total} $k$-jet curvature, i. e. the line bundle $\OO_{X_k}(1)$ is big on $X_k$ and the singularity set $\Sigma_{h_k}$ is contained in $X_k^{\sing} \subset X_k$ (see Definition \ref{defi:KJet}). Then the canonical line bundle $K_X$ is ample.
\end{thm}

Recent works \cite{BD15, Bro16, Den17} show that there exist many examples of varieties satisfying the aforementioned condition, such that general hypersurfaces $H_d \subset X$ of high degree and general complete intersections $H_{d_1, \ldots, d_c} \subset X$ of at least $c \geqslant \lceil n/2 \rceil$ high-degree hypersurfaces. Moreover, a remarkable result of Demailly (see e.g \cite{Dem12}, Corollary 15.74) asserts that for every variety of general type there exist $m, k \in \mathbb{N}$ and an ample line bundle $A$ such that $\OO_{X_k}(m) \otimes (- \frac{m}{kn}(1 + \frac12 + \ldots + \frac1k)A)$ is big, so that $(X, T_X)$ has negative total $k$-jet curvature. We give a converse to this theorem (see Proposition \ref{prop:Big}). Note also that from the results of \cite{DR15} it follows that for many hyperbolic varieties (for example, products $C_1 \times \cdots \times C_n$ of genus $g \geqslant 2$ curves) the intersection of the base loci of jet differentials dominate $X$. Therefore the new methods have to be found to answer the Question \ref{que:KJet} in full generality.

The paper is organized as follows. In Section \ref{se:Jets} we recall basic definitions and results on jet bundles, jet differentials and curvature. The proof of Theorem \ref{thm:Main} is given in Section \ref{se:Proof}. Finally, in Section \ref{se:Concl} we discuss versions of Theorem \ref{thm:Main} in the case of a log-pair $(X, D)$ and in the case of singular directed variety $(X, V)$.

\emph{Acknowledgement.} The author is grateful to J.-P.Demailly for posing Question \ref{que:KJet} and for his advice, to Simone Diverio for his valuable explanations and for pointing out the reference \cite{CamP15}. Also the author thanks the organizers of the conference "Complex Analytic and Differential Geometry" in honor of J.-P.Demailly, for stimulating athmosphere. Last but not least, the author is grateful to Misha Verbitsky for his support.

\section{Jet bundles and curvature}\label{se:Jets}

In this section we introduce the definitions and basic results on jet bundles and jet differentials following \cite{Dem97}. 

\subsection{Jets of curves and jet differentials}

 A {\em directed manifold} is a pair $(X, V)$ where $X$ is a compact complex manifold of dimension $n$ and $V$ is a linear subspace of generic rank $r$ of the tangent bundle $T_X$. In fact, below we will need only the {\em absolute} case, so the reader can suppose that $V = T_X$ and $r = \mathrm{rk}(T_X) = n$. We consider the bundle $p_k \colon J_kV \to X$ of $k$-jets of parameterized curves in $X$. It is the set of equivalence classes of holomorphic maps $f \colon (\mathbb{C}, 0) \to (X, x)$ tangent to $V$ with $f \sim g$ if and only if all derivatives $f^{(j)}(0) = g^{(j)}(0)$ for $0 \leqslant j \leqslant k$ coincide when computed in some local coordinate chart near $x$. The projection map is defined by $p_k([f]) = f(0)$. Choose local holomorphic coordinates $(z_1, \ldots ,z_n)$ on an open set $\Omega \subset X$; then for every $x \in \Omega$ the fibers $J_kV_x$ can be seen as $\mathbb{C}^r$-valued maps $$f = (f_1, \ldots, f_r) \colon (\mathbb{C}, 0) \to \Omega \subset X$$ that are determined by their Taylor expansion at $t = 0$ $$ f(t) = x +tf'(0) + \frac{t^2}{2!}f''(0) + \ldots + \frac{t^k}{k!}f^{(k)}(0) + O(t^{k + 1}).$$ Thus in these coordinates the fibers are parameterized by $$ \left((f'_1(0), \ldots ,f'_r(0)); (f''_1(0), \ldots ,f''_r(0)); \ldots; (f^{(k)}_1(0), \ldots ,f^{(k)}_r(0))\right) \in (\mathbb{C}^{r})^{k}$$ so that $J_kV$ is a locally trivial $\mathbb{C}^{rk}$-bundle over $X$. We denote by $J_kV^{\reg}$ the open subset of {\em regular} $k$-jets $$J_kV^{\reg} = \{ [f] \in J_kV \mid f'(0) \neq 0\}.$$ We define $\Gk$ to be the group of germs of biholomorphic maps $\varphi \colon (\mathbb{C}, 0) \to (\mathbb{C}, 0)$ given by $$ t \mapsto \varphi(t) = a_1t + a_2t^2 + \ldots + a_kt^k, \qquad a_1 \in \mathbb{C}^*, a_i \in \mathbb{C}, 2 \leqslant i \leqslant k$$ with the composition taken modulo terms $t^i$ for $i > k$. Then $\Gk$ is a $k$-dimensional nilpotent complex Lie group, which admits a fiberwise action on $J_kV$ (and on $J_kV^{\reg}$) by change of parameter $(f, \varphi) \to f \circ \varphi$. There is an exact sequence $$ 1 \to \Gpk \to \Gk \to \mathbb{C}^* \to 1$$ the map $\Gk \to \mathbb{C}^*$ being $\varphi \mapsto \varphi'(0)$. The corresponding $\mathbb{C}^*$-action is simply the weighted action $$\lambda \cdot (f', \ldots ,f^{(k)}) = (\lambda f', \lambda^2f'', \ldots , \lambda^kf^{(k)}).$$ Notice also that $\Gk$ has a representation in $GL_k(\mathbb{C})$ given in the above basis by 
$$\varphi(t) = a_1t + \ldots + a_kt^k \longmapsto \begin{pmatrix} 
a_1 & a_2 & a_3 & \ldots & a_k \\
0 & a^2_1 & 2a_1a_2 & \ldots & 2a_1a_{k-1} + \cdots \\
0 & 0 & a^3_1 & \ldots & 3a^2_1a_{k-2} + \cdots \\
0 & 0 & 0 & \ldots & \cdot \\
\cdot & \cdot & \cdot & \ldots & \cdot \\
0 & 0 & 0 & \ldots & a^k_1
\end{pmatrix} \in GL_k(\mathbb{C})$$ the $(i, j)$-th entry of the matrix being $$ p_{ij} = \sum_{l_1 + l_2 + \ldots + l_i = j}a_{l_1}a_{l_2} \cdots a_{l_i}.$$ We are interested in polynomial functions $P(f', \ldots ,f^{(k)})$ of weighted degree $l_1 + 2l_2 + \ldots + kl_k = m$ on fibers of $J_kV$ invariant under the $\Gk$-action: $$ P\left((f \circ \varphi)', (f \circ \varphi)'', \ldots , (f \circ \varphi)^{(k)}\right) = \varphi'(0)^mP(f', f'', \ldots ,f^{(k)}).$$ These functions are called {\em invariant jet differentials of order $k$ and degree $m$} and form a vector bundle $E_{k, m}V^* \to X$ (\cite{Dem97}, Definition 6.7). Invariant jet differentials can be interpreted as sections of line bundles on a certain compactification of $J_kV^{\reg} / \Gk$; this compactification is introduced in the next subsection.

\subsection{Projectivized jet bundles}

In this subsection $(X, V)$ is a directed manifold with $ r = \mathrm{rk}(V) \geqslant 2$.

\begin{defi} The {\em $k$-th (stage of) the Demailly -- Semple tower} of a directed manifold $(X, V)$ is a directed manifold $(X_k, V_k)$ defined inductively by $(X_0, V_0) := (X, V)$ and $$ (X_k, V_k) := \left(\mathbb{P}(V_{k-1}), (\pi_{k-1,k})_*^{-1}\OO_{X_k}(- 1)\right)$$ where $\pi_{k-1,k} \colon X_k \to X_{k-1}$ is the projection, $(\pi_{k-1,k})_* = d\pi_{k-1,k} \colon T_{X_k} \to (\pi_{k-1,k})^{-1}T_{X_{k-1}}$ its differential and $(\pi_{k-1,k})_*^{-1}\OO_{X_k}(- 1)$ is a linear subspace of $T_{X_k}$ which can be defined pointwise by $$ V_k = (\pi_{k-1,k})_*^{-1}\OO_{X_k}(- 1) = \{ \xi \in T_{X_k, (x, [v])} \mid (\pi_{k-1,k})_*\xi \in \mathbb{C}v\}. $$ 
\end{defi}

We get for every $k \in \mathbb{N}$ a tower of $\mathbb{P}^r$-bundles over $X$ with $\dim(X_k) = n + k(r - 1)$ and $\mathrm{rk}(V_k) = r$. For all pair of indices $0 \leqslant j \leqslant k$ we have natural morphisms $\pi_{j,k} \colon X_k \to X_j$ and their differentials $$(\pi_{j,k})_* = (d\pi_{j,k})|_{V_k}\colon V_k \to (\pi_{j,k})^{-1}V_j.$$ The manifolds $X_k$ carry hyperplane line bundles $\OO_{X_k}(1)$. We introduce the following notation for weighted line bundles: $$ \OO_{X_k}(\mathbf{a}) := \bigotimes^{k}_{i = 1}\pi^*_{i,k}\OO_{X_i}(a_i) \qquad \mathbf{a} = (a_1, \ldots ,a_k) \in \mathbb{Z}^k.$$ For every germ $f \colon \mathbb{C} \to (X, V)$ of holomorphic curves we can define inductively the $k$-th lifting $$ f_{[k]}(t) := (f_{[k-1]}(t), [f'_{[k-1]}(t)]) $$ of $f \colon \mathbb{C} \to (X, V)$ to $f_{[k]} \colon \mathbb{C} \to (X_k, V_k)$. Denote by $X^{\reg}_k$ the set of points of $X_k$ which can be reached by liftings of regular germs of curves. We have the following result (\cite{Dem97}, Theorem 6.8).

\begin{prop}\label{prop:Compact} Let $(X, V)$ be a directed manifold, $r = \mathrm{rk}(V) \geqslant 2$. Consider the bundle $J_kV^{\reg}$ of regular $k$-jets of curves $f \colon \mathbb{C} \to X$. Then there exists a holomorphic embedding $J_kV^{\reg}/\Gk \hookrightarrow X_k$ over $X$ which identifies $J_kV^{\reg}/\Gk$ with $X^{\reg}$. In other words, the manifold $X_k$ is a relative compactification of $J_kV^{\reg}/\Gk$ over $X$ and $X^{\sing}_k = X_k \setminus X^{\reg}_k$ is a vertical divisor in $X_k$. Moreover, we have a direct image formula $$ (\pi_{0,k})_*\OO_{X_k}(m) \simeq E_{k,m}V^*$$ and for every line bundle $L \to X$ we have an identification $$ H^0(X_k, \OO_{X_k}(m) \otimes \pi_{0,k}^{*}L) = H^0(X, E_{k,m}V^* \otimes L).$$ More precisely, for every $\mathbf{a} = (a_1, \ldots , a_n) \in \mathbb{Z}^k$ we have $$ (\pi_{0,k})_*\OO_{X_k}(\mathbf{a}) = \OO(F_{\mathbf{a}}E_{k,m}V^*) $$ where $F_{\mathbf{a}}E_{k,m}V^*$ is the subbundle of polynomials $P(f', f'', \ldots, f^{(k)}) \in E_{k,m}V^*$ involving only monomials $(f^{(\bullet)})^l$ such that $$ l_{s +1} + 2l_{s+2} + \ldots + (k-s)l_{k} \leqslant a_{s+1} + \ldots + a_{k} $$ for every $0 \leqslant s \leqslant k -1$.
\end{prop}

We obtain a filtration $F_{\mathbf{a}}$ on $E_{k,m}V^*$ with the associated graded object described in \cite{Dem97}, \textsection 12.

\begin{prop}\label{prop:Decomp} The graded object of the above filtration on $E_{k,m}V^*$ is a direct sum of irreducible $GL(V)$-representations and is isomorphic to $$ \mathrm{Gr}^{\bullet}(E_{k,m}V^*) = \left(\bigoplus_{l \in \mathbb{N}^k,  l_1 + 2l_2 + \ldots + kl_k = m}S^{l_1}V^* \otimes S^{l_2}V^* \otimes \cdots \otimes S^{l_k}V^*\right)^{\Gpk}.$$
\end{prop}

\subsection{Curvature of jet bundles and hyperbolicity}

\begin{defi}\label{defi:KJet} Let $(X, V)$ be a directed manifold. We say that $(X, V)$ has {\em negative $k$-jet curvature} if there exist a singular Hermitian metric on the line bundle $\OO_{X_k}(1)$ such that $$ \Theta_{h_k}(\OO_{X_k}(1))|_{V_k} \geqslant \varepsilon \omega|_{V_k} \qquad V_k \subset T_{X_k}$$ in the sense of currents for some metric $\omega$ on $X_k$ and $\varepsilon > 0$. We say that $(X, V)$ has {\em negative total $k$-jet curvature} if $$ \Theta_{h_k}(\OO_{X_k}(1)) \geqslant \varepsilon \omega.$$ Finally, we say that $(X, V)$ has {\em nondegenerate} negative (total) $k$-jet curvature if in addition the degeneration set $\Sigma_{h_k}$ of the metric $h_k$ lies in $X^{\sing}_k$.
\end{defi}

A connection between hyperbolicity and $k$-jet negativity is given by the Ahlfors-Schwarz lemma (see \cite{Dem97}, Theorem 7.8).

\begin{prop} Let $(X, V)$ be a compact directed manifold. Suppose that $(X, V)$ has a metric with negative $k$-jet curvature. Then for every entire curve $f \colon \mathbb{C} \to (X, V)$ the $k$-th lifting $f_{[k]} \colon \mathbb{C} \to X_k$ is such that $f_{[k]}(\mathbb{C}) \subset X^{\sing}_k$. In particular, if $(X, V)$ has a metric with nondegenerate negative $k$-jet curvature then $(X, V)$ is hyperbolic.
\end{prop}

In particular, negativity of total $k$-jet curvature is equivalent to bigness of $\OO_{X_k}(1)$, i.e. existence of many invariant jet differentials. In this case we have the following fundamental theorem (see e.g. \cite{Dem97}, Corollary 7.9) which establishes the existence of many algebraic differential equations satisfied by every entire curve $f \colon \mathbb{C} \to (X, V)$.

\begin{thm}{(M. Green -- P. Griffiths, J.-P. Demailly, Y.-T. Siu -- S.-K. Yeung)} Let $(X, V)$ be a compact directed manifold. Suppose that there exist $k, m \in \mathbb{N}$ and an ample line bundle $A$ such that $H^0(X_k, \OO_{X_k}(m) \otimes \pi^*_{0,k}A^{-1}) = H^0(X, E_{k,m}V^* \otimes A^{-1})$ has nonzero sections $\sigma_1, \ldots , \sigma_N$. Denote by $Z$ the base locus of these sections. Then for every entire curve $f \colon \mathbb{C} \to (X, V)$ the image of its $k$-th lifting lies in $Z$: $f_{[k]}(\mathbb{C}) \subset Z$. In other words, every entire curve $f \colon \mathbb{C} \to (X, V)$ satisfies the algebraic differential equations $\sigma_1(f', f'', \ldots, f^{(k)}) = \ldots = \sigma_N(f', f'', \ldots, f^{(k)}) = 0$.
\end{thm}

\section{The main result}\label{se:Proof}

In this section we prove our main Theorem \ref{thm:Main}. The proof is divided into three steps. Below we assume the following to hold.

\begin{assum}\label{assum:Main} Let $X$ be a smooth complex projective variety. We assume that there exists $k \in \mathbb{N}$ such that $\OO_{X_k}(1)$ is endowed with a Hermitian metric $h_k$ with degeneration set $\Sigma_{h_k} \subset X_k^{\sing}$ and curvature current $\Theta_{h_k}(\OO_{X_k}(1)) \geqslant \varepsilon \omega$ for a metric $\omega$ on $X_k$ and $\varepsilon > 0$.
\end{assum}

\subsection{Step 1: the canonical line bundle is nef}

We present here a proposition from \cite{Dem97}, 8.1, which implies nefness of $K_X$.

\begin{prop}\label{prop:Curves} Let $X$ be a complex projective variety. Suppose that $(X, T_X)$ has negative $k$-jet curvature. Then there exist a constant $\varepsilon > 0$ such that for every closed irreducible curve $C \subset X$ such that $\nu_{[k]}(\bar C) \not \subset \Sigma_{h_k}$ the following inequality holds: 
$$ -\chi(\bar C) = 2g(\bar C) - 2 \geqslant \varepsilon\deg_{\omega}(C) + \sum_{t \in \bar C}(m_{k-1}(t) - 1) > 0 $$
 where $\nu \colon \bar C \to C$ is the normalization and $m_{k}(t)$ is the multiplicity at point $t$ of the lifting $\nu_{[k]} \colon \bar C \to P_kT_X = X_k$.
\end{prop}

\begin{proof} We have a lifting $\nu_{k} \colon \bar C \to X_k$ and the derivative gives a canonical map $$\nu'_{[k-1]} \colon T_{\bar C} \to \nu^*_{[k]}\OO_{X_k}(- 1).$$ Let $t_j \in \bar C$ be the singular points of $\nu_{[k-1]}$ and let $m_j = m_{k-1}(t_j)$ be the corresponding multiplicities. Then $\nu'_{[k-1]}$ vanishes to order $m_j - 1$ at $t_j$ so we obtain $$ T_{\bar C} \simeq \nu^*_{[k]}\OO_{X_k}(- 1) \otimes \OO_{\bar C}\Bigl( - \sum_j(m_j - 1)t_j\Bigr).$$ We take a metric $h_k$ with negative $k$-jet curvature and degeneration set $\Sigma_{h_k}$. From the nondegeneracy assumption it follows that $\nu_{[k]}(\bar C) \not \subset \Sigma_{h_k}$ so $\int_{\bar C}\Theta_{h_k}(\nu^*_{[k]}\OO_{X_k}(1))$ is well-defined. The Gau\ss  -- Bonnet formula yields $$2g(\bar C) - 2 = \int_{\bar C}\Theta(\Omega^1_{\bar C}) = \sum_j(m_j - 1) + \int_{\bar C}\Theta_{h_k}(\nu^*_{[k]}\OO_{X_k}(1)).$$ The hypothesis on curvature implies $$ \Theta_{h_k}(\nu^*_{[k]}\OO_{X_k}(1))(\xi) \geqslant \varepsilon'|\xi|^2_{\omega_k} \geqslant \varepsilon''|(\pi_{0,k})_*\xi|^2_{\omega}$$ for some $\varepsilon', \varepsilon'' > 0$ and smooth Hermitian metrics $\omega_k$ on $X_k$. As $\pi_{0,k} \circ \nu_{[k]} = \nu$ we infer that $\Theta_{h_k}(\nu^*_{[k]}\OO_{X_k}(1)) \geqslant \varepsilon\omega$ and therefore $$ \int_{\bar C}\Theta_{h_k}(\nu^*_{[k]}\OO_{X_k}(1)) \geqslant \frac{\varepsilon''}{2\pi}\int_{\bar C}\omega = \varepsilon\deg_{\omega}(C)$$ with $\varepsilon = \varepsilon''/2\pi$. Proposition \ref{prop:Curves} now follows.
\end{proof}

\begin{cor} In the assumptions of \ref{assum:Main} , the canonical line bundle $K_X$ is nef.
\end{cor}

\begin{proof} Indeed, by Proposition \ref{prop:Curves} we have the inequality $\chi(\bar C) < 0$ for every $C \subset X$ which shows that $X$ has no rational curves, therefore $K_X$ is nef by the Cone Theorem (see e. g. \cite{Deb01}, Theorem 6.1).
\end{proof}

\subsection{Step 2: the canonical line bundle is big}

This is the core of the proof. First of all, we state the bigness criterion of Campana and P\u{a}un (\cite{CamP16}, Theorem 1.2).

\begin{thm}\label{thm:CampanaPaun}{(F. Campana -- M. P{\u{a}}un)} Let $X$ be a complex projective variety such that $K_X$ is pseudoeffective. If there exists a big line bundle $\LL \to X$ and $m \in \mathbb{N}$ such that $$ H^0\left(X, \Omega_X^1(D)^{\otimes m}\otimes \LL^{-1}\right) \neq 0$$ then the canonical line bundle $K_X + D$ of the pair $(X, D)$ is big.
\end{thm}

Thus in our case we need to construct a nontrivial morphism $\LL \to (\Omega^1_X)^{\otimes m}$ from a big line bundle $\LL$ to the sheaf of pluridifferentials $(\Omega^1_X)^{\otimes m}$. To do so, we will use the fitration structure on $\OO(E_{k,m})$ as follows (see \cite{CamP15}, Theorem 4.5 for the version for Green-Griffiths jets).

\begin{prop}\label{prop:Big} Let $X$ is a projective variety with the canonical line bundle $K_X$ pseudoeffective. Suppose that there exists $\mathbf{a} = (a_1, \ldots, a_k) \in \mathbb{N}^k$ such that $\OO_{X_k}(\mathbf{a})$ is big. Then the canonical line bundle $K_X$ is big.
\end{prop}

\begin{proof}
Let us denote the direct image sheaf $(\pi_{0,k})_*\OO_{X_k}(\mathbf{a})$ by $\EE$ with $\mathbf{a} = (a_1, \ldots, a_k) \in \mathbb{N}^k$ such that $\OO_{X_k}(\mathbf{a})$ is big. Take also an ample line bundle $A$ on $X$. By Proposition \ref{prop:Compact} the direct image $(\pi_{0,k})_*$ gives the identification $$ H^0(X_k, \OO_{X_k}(k \cdot \mathbf{a}) \otimes A^{-1}) = H^0(X, \Sym^k(\EE) \otimes A^{-1})$$ for all $k \in \mathbb{N}$. Since $\OO_{X_k}(\mathbf{a})$ is big, the Kodaira lemma gives a nonzero section $s \in H^0(X, \Sym^k\EE \otimes A^{-1})$ for some $k \in \mathbb{N}$ or, equivalently, a sheaf injection $$0 \to A \to \Sym^k(\EE) \to \bigotimes_k \OO(E_{k,m}\Omega^1_X).$$ By Proposition \ref{prop:Decomp} there exists a filtration on $\OO(E_{k, m}\Omega_X^1)$ such that the graded object is a direct sum of irreducible $GL(T_X)$-representations contained in $(\Omega^1_X)^{\otimes l}$ for $l \leqslant |\mathbf{a}|$. By induction on $k$, we can endow $k$-th tensor power of $\OO(E_{k,m}\Omega^1_X)$ with a filtration such that its graded pieces are contained in $(\Omega^1_X)^{\otimes lk}$. Taking the induced morphism of graded sheaves, we obtain a nontrivial morphism $A \to (\Omega^1_X)^{\otimes lk}$ for some $l \in \mathbb{N}$. Applying Theorem \ref{thm:CampanaPaun} with $D = \emptyset$, $\LL = A$ and $m = lk$ we obtain bigness of $K_X$.
\end{proof}

Recall that by Step 1 the canonical line bundle $K_X$ is nef. Therefore an application of Proposition \ref{prop:Big} with $\mathbf{a} = (0, \ldots, 1)$ and completes Step 2 of the proof. 

\begin{rmk}\label{rmk:Alter} If the degeneracy locus $\Sigma$ of the metric on $\OO_{X_k}(\mathbf{a})$ does not dominate $X$ then $K_X$ is automatically pseudoeffective. Indeed, we only need to note that by Proposition \ref{prop:Curves} $X$ is not covered by rational curves, so $K_X$ is pseudoeffective by \cite{BDPP13}. 
\end{rmk}

\subsection{Step 3: the canonical line bundle is ample} This part of the proof is completely standard (see e. g. \cite{WY16, DT16, HLW10}). Now $K_X$ is big and nef, so by the Base Point-Free theorem it is semiample, i. e. there exists a birational contraction $c \colon X \to X'$ such that $K_X = c^*A$ where $A$ is ample on $X'$. Every irreducible component of $\Exc(c) = \mathrm{Null}(K_X)$ must be covered by rational curves contracted by $c$. On the other hand, by Proposition \ref{prop:Curves} there are no rational curves on $X$. Thus $\Exc(c) = \mathrm{Null}(K_X)$ is empty. Then by Nakai -- Moishezon criterion $K_X$ is ample, as desired. Alternatively, we could use the following result \cite{Tak08}, Theorem 1.1 (ii).

\begin{thm}\label{thm:Takayama}{(S. Takayama)} Let $X$ be a smooth projective variety with $K_X$ big. Then every irreducible component of the non-ample locus $\mathrm{NAmp}(K_X)$ is uniruled.
\end{thm}

Again, by Proposition \ref{prop:Curves} and Theorem \ref{thm:Takayama} the canonical line bundle $K_X$ is ample. The proof of Theorem \ref{thm:Main} is now complete.

\section{Concluding remarks}\label{se:Concl}

\subsection{The logarithmic case} Consider a pair $(X, D)$ where $D$ is a normal crossings divisor. To study hyperbolicity of the open variety $X \setminus D$ we can use the techniques of {\em logarithmic} jet bundles and jet differentials as developed in \cite{DL01}. In particular we can consider the {\em log-directed manifold} $(X, D, T_X(D))$ and construct the Demailly -- Semple tower $$ (X_k, D_k, V_k) \to (X_{k-1}, D_{k-1}, V_{k-1}) \to \ldots (X, D, T_X(D)).$$ It is possible to extend Conjecture \ref{conj:Kob} and Question \ref{que:KJet} to the logarithmic setting (see \cite{LZ17} for recent progress towards Conjecture \ref{conj:Kob}). By adapting the arguments in Section \ref{se:Proof} we can prove the following version of our main Theorem \ref{thm:Main}.

\begin{thm}\label{thm:Log} Let $X$ be a smooth projective variety and $D$ a normal crossings divisor on $X$. Suppose that the log-directed manifold $(X, D, T_X(D))$ has a metric with negative total $k$-jet curvature such that the degeneration set $\Sigma_{h_k}$ does not dominate $X$. Then the canonical line bundle $K_X + D$ of the pair $(X, D)$ is big. If. moreover, we have $\Sigma_{h_k} \subset X^{\sing}_k \cup D_k$ and $D$ does not contain rational curves then $K_X + D$ is ample.
\end{thm}

\subsection{The case of a singular directed variety} Consider now the general case of a directed variety $(X, V)$ where $V \subset T_X$ is a holomorphic subbundle outside of a codimension $\geqslant 2$ analytic subvariety $\Sing(V)$. In this case the canonical sheaf can be defined by $K_V = i_*\det i^*(\OO(V^*))$ for $i \colon X\setminus \Sing(V) \to X$ an inclusion. This definition is standard in foliation theory. On the other hand, a more refined notion of the canonical sheaf $\mathscr{K}_V$ was introduced by Demailly (see e.g. \cite{Dem15}, Definition 2.10). This version of the canonical sheaf better reflects the impact that the singularities of $V$ have on the geometry of $(X, V)$. A possible generalization of Question \ref{que:KJet} is the following.

\begin{que} Suppose that $r = \mathrm{rk}(V) \geqslant 2$ and $(X, V)$ has a metric with nondegenerate negative (total) $k$-jet curvature. Is the canonical sheaf $\mathscr{K}_V$ big in the sense of Demailly?
\end{que}

We hope to be able to address this question in a forthcoming paper.

\end{document}